\documentclass[11pt]{amsart}
\usepackage{amssymb,amsmath,amsfonts,amscd,euscript}
\usepackage{hyperref}
\newcommand{\nc}{\newcommand}

\numberwithin{equation}{section}
\newtheorem{thm}{Theorem}[section]
\newtheorem{prop}[thm]{Proposition}
\newtheorem{lem}[thm]{Lemma}
\newtheorem{cor}[thm]{Corollary}
\newtheorem{rem}[thm]{Remark}
\newtheorem{definition}[thm]{Definition}
\newtheorem{example}[thm]{Example}
\newtheorem{dfn}[thm]{Definition}

\newtheorem{conj}[thm]{{\bf Conjecture}}

\nc{\gl}{\mathfrak{gl}}
\nc{\GL}{\mathfrak{GL}}
\nc{\g}{\mathfrak{g}}
\nc{\gh}{\widehat\g}
\nc{\h}{\mathfrak{h}}
\nc{\la}{\lambda}
\nc{\al}{\alpha }
\nc{\be}{\beta }
\nc{\ve}{\varepsilon }
\nc{\om}{\omega }

\nc{\ta}{\theta}
\nc{\ch}{{\mathop {\rm ch}}}
\nc{\Tr}{{\mathop {\rm Tr}\,}}
\nc{\Id}{{\mathop {\rm Id}}}
\nc{\ad}{{\mathop {\rm ad}}}
\nc{\bra}{\langle}
\nc{\ket}{\rangle}
\nc{\x}{{\bf x}}
\nc{\bm}{{\bf m}}
\nc{\bs}{{\bf s}}
\nc{\br}{{\bf r}}
\nc{\bb}{{\bf b}}
\nc{\bk}{{\bf k}}
\nc{\bp}{{\bf p}}
\nc{\pa}{\partial}
\nc{\ld}{\ldots}
\nc{\cd}{\cdots}
\nc{\hk}{\hookrightarrow}
\nc{\T}{\otimes}
\nc{\gr}{\mathrm{gr}}
\nc{\ov}{\overline}

\nc{\cO}{\mathcal O}
\nc{\msl}{\mathfrak{sl}}
\nc{\mgl}{\mathfrak{gl}}
\nc{\U}{\mathrm U}
\nc{\V}{\EuScript V}
\nc{\cL}{\mathcal{L}}
\nc{\Res}{\mathrm{Res\ }}

\newcommand{\bC}{{\mathbb C}}
\newcommand{\bQ}{{\mathbb Q}}
\newcommand{\bZ}{{\mathbb Z}}

\newcommand{\bP}{{\mathbb P}}
\newcommand{\bL}{{\mathbb L}}

\newcommand{\fh}{{\mathfrak h}}

\newcommand{\fg}{{\mathfrak g}}

\newcommand{\fb}{{\mathfrak b}}

\newcommand{\fn}{{\mathfrak n}}

\newcommand{\eM}{\EuScript{M}}

\nc{\Q}{\mathfrak Q}
\newcommand{\gW}{{\mathbb{W}}}

\begin{document}

\title[Vertex algebras and coordinate rings of semi-infinite flags]
{Vertex algebras and coordinate rings of semi-infinite flags}

\author{Evgeny Feigin}
\address{Evgeny Feigin:\newline
Department of Mathematics,\newline
National Research University Higher School of Economics,\newline
Usacheva str. 6, 119048, Moscow, Russia,\newline
{\it and }\newline
Skolkovo Institute of Science and Technology, Skolkovo Innovation Center, Building 3,
Moscow 143026, Russia
}
\email{evgfeig@gmail.com}

\author{Ievgen Makedonskyi}
\address{Ievgen Makedonskyi:\newline
Department of Mathematics, Kyoto University, Oiwake,
Kita-Shirakawa, Sakyo Kyoto 606--8502, Japan
\newline
{\it and} \newline
Department of Mathematics,\newline
National Research University Higher School of Economics,\newline
Usacheva str. 6, 119048, Moscow, Russia
}
\email{makedonskii\_e@mail.ru}

\begin{abstract}
The direct sum of irreducible level one integrable representations of affine Kac-Moody Lie algebra of (affine) type $ADE$
carries a structure of $P/Q$-graded vertex operator algebra.
There exists a filtration on this direct sum studied by Kato and Loktev such that the corresponding graded vector space is a direct sum of
global Weyl modules. The associated graded space with respect to the dual filtration is isomorphic to the homogenous coordinate ring
of semi-infinite flag variety. We describe the ring structure in terms of vertex operators and endow the homogenous coordinate ring
with a structure of $P/Q$-graded vertex operator algebra. We use the vertex algebra approach to derive semi-infinite
Pl\"ucker-type relations in the homogeneous coordinate ring.
\end{abstract}

\maketitle

\section*{Introduction}
Let $\fg$ be a simple complex Lie algebra of type $ADE$ (i.e. simply-laced). Let $P$ and $Q$ be the weight and root lattices
of $\fg$. We denote by $\gh$ the corresponding affine Kac-Moody Lie algebra, which is the central extension of the loop
algebra $\fg\T\bC[t,t^{-1}]$. It is well known (see e.g. \cite{C,Kac1,Kum}) that the level one integrable irreducible
$\gh$-modules are in one-to-one
correspondence with the quotient $P/Q$. In particular, the neutral element in $P/Q$ corresponds to the
basic (vacuum) representation $L(\Lambda_0)$. We note that if $\fg=\msl_n$, then there are $n$ level one integrable
$\gh$-modules: $L(\Lambda_0)$ and one module for each fundamental weight of $\fg$. And if $\fg$ is of type $E_8$, then $P=Q$ and
the only integrable level one module is the basic one.

The basic representation $L(\Lambda_0)$ of $\gh$ carries a structure of vertex operator algebra (VOA) \cite{FK,FB,Kac2}.
In short, a VOA is a correspondence between vectors in $L(\Lambda_0)$ and fields on the same space (the formal
Taylor series with coefficients in ${\rm End}(L(\Lambda_0))$). In particular, the action of the generators of $\gh$
is expressed in terms of the components of certain fields. For our purposes we need a generalization of the standard
definition of vertex operator algebra, the so called $\Gamma$-graded VOA for a finite group $\Gamma$
(this notion generalizes the notion of super vertex algebra) (see
\cite{DL,R,BK,S}). We note that we only use the case $\Gamma=P/Q$. The space of states $\bL$ of the $P/Q$-graded VOA we are working with
is equal to the direct sum of all irreducible level one integrable $\gh$-modules (in particular, $\bL$
is naturally graded by the group $P/Q$).  To an element $A\in \bL$ one
attaches a field $Y(A,z)$ acting on $\bL$. An important difference between the $\Gamma$-graded VOAs and their classical
analogues is that the fields $Y(A,z)$ contain rational (not necessarily integer) powers of $z$.
We use these fields in order to study certain structure on level one modules
which involves Weyl modules: a class of representations of the current algebra $\fg\T\bC[t]$.

The current algebra $\fg\T\bC[t]$ can be seen as an intermediate step between $\fg$ and $\gh$. In particular, its
representation theory (algebraic, geometric and combinatorial) has various deep connections with representation theory
of $\fg$ and that of $\gh$ (see e.g. \cite{CL,FL,FM1,KL}). The current algebra has two important classes of cyclic
representations: local Weyl modules $W_\la$ and
global Weyl modules $\gW_\la$ (\cite{CP,CFK,FMO}). Both classes are labeled by the same parameter $\la\in P_+$ (the dominant cone
of the weight lattice),
but the local modules are finite-dimensional, while the global Weyl modules are infinite-dimensional. These modules
have many interesting features, in particular, they are important for the study of nonsymmetric Macdonald polynomials
\cite{Ch,CI,I,FKM,Kat, ChK} and semi-infinite flag varieties \cite{BF1,BF2,FMi,KNS}.

In a recent paper \cite{KL} the authors studied a filtration $F_\la\subset \bL$ labeled by dominant weights
$\la\in P_+$, where each $F_\la$ is a thick Demazure module (a subspace generated from an extremal weight $\la$ vector
of $\bL$ by the action of $\fg\T\bC[t^{-1}]$). The main result of \cite{KL} for level one representations claims
that the associated graded ${\rm gr} F_\bullet$ is isomorphic to the direct sum of global Weyl modules
$\bigoplus_{\la\in P_+} \gW_\la$ (see also \cite{ChF}). We consider the dual filtration $F^\dag_\la$ on $\bL$
(with respect to the Shapovalov form)
with the property that ${\rm gr}F^\dag_\bullet$ is isomorphic to the space $\gW^*=\bigoplus_{\la\in P_+} \gW^*_\la$.

The advantage of the dual picture is that the space $\gW^*$ carries an important structure of commutative associative algebra
coming from the $\fg\T\bC[t]$ equivariant surjections $\gW^*_{\la}\T \gW^*_{\mu}\to \gW^*_{\la+\mu}$ \cite{Kat,KNS}.
The algebra  $\gW^*$ is the homogeneous coordinate ring of the semi-infinite flag variety of the group $G[[t]]$ with
respect to the Drinfeld-Pl\"ucker embedding (see \cite{FMi,BF1,BF2}). Our goal is to use the $P/Q$-graded vertex operators
structure on $\bL$ in order to study the algebra $\gW^*$.

Let $Y(A,z)=\sum_{n\in\bQ} A_{(n)}z^{-n-1}$ be the Fourier expansion of a vertex operator. We first prove the following
claim:
\[
A_{(n)} B\in F^\dag_{\la+\mu} \text{ for all } A\in F^\dag_\la, B\in F^\dag_\mu, n\in\bQ.
\]
We conclude that the algebra $\gW^*$ carries structure of $P/Q$-graded vertex operator algebra.

The next goal is to describe the filtration $F^\dag_\la$ in terms of the action of vertex operators. In order to state
the result, we consider the natural embeddings $W_{\omega_i}^*\subset\gW^*_{\omega_i}$ for all fundamental weights $\omega_i$.
Using ${\rm gr}F^\dag_\la\simeq \gW^*_\la$, we fix a lift of the fundamental local Weyl modules into $F^\dag_{\om_i}$
and consider $W_{\omega_i}^*$ as sitting inside $F^\dag_{\om_i}\subset \bL$.
Let $\la=\om_{j_1}+\dots+\om_{j_s}$. We prove that
\[
F_\la^\dag={\rm span}\{A^1_{(n_1)}\dots A^s_{(n_s)}|0\rangle, A^i\in W_{\omega_{j_i}}^*\},
\]
where $|0\rangle\in L(\Lambda_0)\subset \bL$ is the vacuum vector (see \cite{A,Li1,Li2} for similar constructions).
As a consequence we endow $\gW^*$ with the structure of $P/Q$-graded VOA.

Our next task is to describe the multiplication in the algebra $\gW^*$ in terms of the vertex operators.
In short, we write down an explicit formula for the multiplication of the elements from $\gW^*_{\omega_i}$
in terms of the products of vertex operators (see \eqref{multiplication} for the explicit formula).
We apply the resulting formula to derive a set of quadratic relations in the algebra $\gW^*$.
We show that in type $A$ these relations are defining and coincide with those given in \cite{FM2}.

The paper is organized as follows. In Section \ref{notation} we collect main notation and introduce main
objects from the theory of simple, current and  affine Lie algebras. We also briefly recall the formalism of semi-infinite
flag varieties and Weyl modules. In Section \ref{gVOA} we collect main definitions from the theory of $\Gamma$-graded vertex
operator algebras. Section \ref{strategy} is devoted to the brief description of our constructions and to the strategy of the proofs.
In Section \ref{Main} we construct vertex operator algebra filtration ($P/Q$-graded analog of the standard filtration)
and write down the multiplication formula in terms of vertex operators. Section \ref{VOAPl} is devoted to the
derivation of the semi-infinite Pl\"ucker-type relations in the homogeneous coordinate ring of the semi-infinite
flag variety. In Section \ref{typeA} we show that in type $A$ the relations derived  in the previous section
are defining.

\section{Semi-infinite flag varieties, Weyl modules and integrable representations}\label{notation}
In this section we collect main definitions and constructions needed in the main body of the paper.
The three main ingredients are the Pl\"ucker-Drinfeld realization of the semi-infinite flag varieties,
Weyl modules and the connection between the global Weyl modules and integrable level one representations
of the affine Kac-Moody Lie algebra described in \cite{KL}.

\subsection{Simple Lie algebras}
Let $\fg$ be a simple Lie algebra of rank $r$ with the Cartan decomposition $\fg=\fb\oplus \fn_-$, where
$\fb=\fn\oplus\fh$ is a Borel subalgebra.
Let $\omega_1,\dots,\omega_r$ and $\al_1,\dots,\al_r$ be fundamental weights and simple roots. We denote by
$P$ and $P_+$ the weight lattice and its dominant cone. In particular, $P_+=\bigoplus_{i=1}^r \bZ_{\ge 0}\om_i$.
Let $Q$ be the root lattice spanned by $\al_i$, $i=1,\dots,r$. We denote by $\Delta=\Delta_+\sqcup \Delta_-$ the set of roots of
$\fg$. For $\al\in \Delta_+$ we denote by $f_\al\in\fn_-$, $e_\al\in\fn$ the corresponding Chevalley generators.
In what follows we use the following partial order on the lattice $P$:
\begin{equation}\label{order}
\mu\leq\lambda\ \text{ iff }\ \lambda-\mu=\sum_{i=1}^ra_i\alpha_i,\ a_i \in \mathbb{Q}_{\ge 0}.
\end{equation}
We note that $\la-\om_i<\la$ for any $\la\in P$, $i=1,\dots,r$. In particular, $0$ is the smallest element in $P_+$.

For a dominant integral weight $\la\in P_+$ we denote by $V_\la$ irreducible $\fg$ module of highest weight $\la$.
In particular, $V_\la$ contains a highest weight vector $v_\la$ such that $\fn v_\la=0$.
Let $(\cdot,\cdot)$ be the Killing form. For a root $\al$ we denote the corresponding coroot by $\al^\vee$.
One has the following (defining) relations in $V_\la$: $f_\al^{(\la,\al^\vee)+1}v_\la=0$.

For a weight $\la\in P_+$ the dual representation $V^*_\la$ is isomorphic to a highest weight module.
We denote the highest weight of this module by $\la^*$. One has the equality $\la^*=-w_0\la$, where
$w_0$ is the longest element of the Weyl group of $\fg$. In particular, for any $i=1,\dots,r$ the weight
$\om_i^*$ is again a fundamental weight.

Let $G\supset B$ be the Lie groups of $\fg$ and $\fb$. Then the flag variety $G/B$ enjoys the Pl\"ucker embedding
$G/B\subset \prod_{i=1}^r \bP(V_{\omega_i})$. The Pl\"ucker embedding gives rise to the family of line bundles
$\cO(\la)$, $\la\in P_+$ on $G/B$ (the tensor products of pull backs of the line bundles $\cO(1)$ on $\bP(V_{\omega_i})$).
The Borel-Weil theorem states that the space of sections $H^0(G/B,\cO(\la))$ is isomorphic to $V_\la^*$. The coordinate ring
of $G/B$ (with respect to the Pl\"ucker embedding) is given by $\bigoplus_{\la\in P_+} V_\la^*$.

\subsection{Weyl modules}
We consider the current algebra $\fg\T\bC[t]$.
For $x\in\fg$, $k\ge 0$ we denote by $xt^k$ the element $x\T t^k\in\fg\T\bC[t]$.
For a dominant integral weight $\la$ let $W_\la$ be the corresponding
local Weyl module of highest weight $\la$ and let $\gW_\la$ be the global Weyl module.
The global Weyl module $\gW_\la$ is cyclic $\fg\T\bC[t]$ module with cyclic vector $w_\la$ of $\fh\T 1$ weight $\la$
and defining relations
\[
\fn\T\bC[t]. w_\la=0,\ (f_\la\T 1)^{(\la,\al^\vee)+1}w_\la=0.
\]
The defining relations for the local Weyl module $W_\la$ differ by the additional relation $\fh\T t\bC[t]. w_\la=0$.
We note that both local and global Weyl modules are graded, where the grading is defined by saying that the degree of
$w_\la$ is zero and $x\T t^i$ increases the degree by $i$. We refer to this grading as $q$-grading.
For a graded $\fg$-module $M=\bigoplus_{k\in\bZ} M[k]$ with
finite-dimensional graded components we denote $\ch_q M=\sum_{k\in \bZ} q^k \ch M[k]$.
Let $(q)_n=\prod_{i=1}^n (1-q^i)$.

We have the following properties of Weyl modules.
\begin{itemize}
\item $W_\la$ is finite-dimensional and $\dim W_\la=\prod_{i=1}^r (\dim W_{\om_i})^{(\la,\al_i^\vee)}$;
\item  $\U(\fg)w_\la\simeq V_\la\subset W_\la$;
\item $\ch \gW_\la=\ch W_\la\prod_{i=1}^r (q)^{-1}_{(\la,\al_i^\vee)}$.
\end{itemize}

\begin{rem}
In type $A$ one has $W_{\om_i}\simeq V_{\om_i}$. However, in general $W_{\om_i}$ is not irreducible as $\fg$ module.
\end{rem}

It was shown in \cite{Kat} that one has $\fg\T\bC[t]$ equivariant embeddings $\gW_{\la+\mu}\subset \gW_\la\T\gW_\mu$,
$w_{\la+\mu}\mapsto w_\la\T w_\mu$. Thus one gets the natural algebra structure on the space
$\bigoplus_{\la\in P_+} \gW_\la^*$.

\subsection{Construction of global Weyl modules}
In this subsection we describe an explicit construction of global Weyl modules and dual global Weyl modules.
Let $U$ be a finite-dimensional $\fg[t]$-module (here and below we use the shorthand notation $\fg[t]=\fg\T\bC[t]$). 
We construct two $\fg[t]$-modules $U[t]$ and $U[t^{-1}]$ satisfying the following properties:
\begin{gather}\label{globalizationproperties1}
\ch U[t]=\frac{1}{1-q}\ch(U),\ \ch U[t^{-1}]=\frac{1}{1-q^{-1}}\ch U;\\ \label{globalizationproperties2}
\text{ (co)cyclicity of } U\text{ implies  (co)cyclicity of } (U[t^{-1}]) U[t];\\
(U^*)[t^{-1}]=(U[t])^*\label{globalizationproperties3}.
\end{gather}

We consider the vector space $U[t,t^{-1}]:=U \otimes \mathbb{C}[t,t^{-1}]$ and define the following action of $\fg[t]$ on this space:
for
any $m\in \bZ_{\ge 0}$, $k\in\bZ$, $x \in \fg$, $u \in U$:

\begin{equation}\label{globalization}
xt^m.u \otimes t^k=\sum_{j=0}^m\binom{m}{j}(xt^ju)\otimes t^{m+k-j}.
\end{equation}

It is indeed an action:

\begin{multline}
(yt^nxt^m-xt^myt^n).u\otimes t^k=\sum_{i=0}^n\binom{n}{i}\sum_{j=0}^m\binom{m}{j}(yt^ixt^ju-xt^jyt^iu)\otimes t^{m+n+k-i-j}\\=
\sum_{i,j=0}^{\infty}\binom{m}{j}\binom{n}{i}([x,y]t^{i+j}u)\otimes t^{m+n+k-i-j}=[x,y]t^{m+n}.u\otimes t^k.
\end{multline}
The last equality holds because $\sum_{i+j=c}\binom{m}{j}\binom{n}{i}=\binom{m+n}{c}$.

\begin{dfn}
We define $U[t]$ as a submodule $U\T\bC[t]$ of $U[t,t^{-1}].$ We define $U[t^{-1}]$ as a quotient mdoule
$U[t,t^{-1}]/U \otimes t\mathbb{C}[t]$ of  $U[t,t^{-1}].$
\end{dfn}

\begin{lem}
Properties \eqref{globalizationproperties1}, \eqref{globalizationproperties2}, \eqref{globalizationproperties3} hold.
\end{lem}
\begin{proof}
The equations on the characters are obvious. To show cyclicity and cocyclicity it  suffices to apply standard linear algebra
arguments (the Vandermonde determinant).
Finally, define the pairing between $U[t]$ and $U^*[t^{-1}]$ in the following way:
\[(u_1 \otimes t^k,u_2 \otimes t^l)=\delta_{k+l,0}(u_1,u_2).\]
This pairing gives a duality of modules.
\end{proof}

\begin{prop}\label{globalweylconstruction}
Let $W_{\omega_i}$ be the fundamental local Weyl module. Then $W_{\omega_i}[t]\simeq \mathbb W_{\omega_i}$,
$W^*_{\omega_i}[t^{-1}]\simeq (\mathbb W_{\omega_i})^*$.
\end{prop}
\begin{proof}
$W_{\omega_i}[t]$ is the cyclic module satisfying all defining relations of the global Weyl module and has the same character. Therefore
these modules are isomorphic. The second claim follows from  \eqref{globalizationproperties3}.
\end{proof}

\subsection{Semi-infinite flag varieties}
We consider the product of projective spaces $\prod_{\la\in P_+} \bP(V_\la[[t]])$. The semi-infinite fag variety
$\Q$ is a subvariety of this product formed by all collections of lines $L_\la\subset V_\la[[t]]$ such that
$L_{\la+\mu}\mapsto L_\la\T L_\mu$ under the natural embedding  $V_{\la+\mu}[[t]]\to V_{\la}[[t]]\T_{\bC[[t]]} V_{\mu}[[t]]$
(\cite{FMi,BG}).
One easily sees that $\Q$ is embedded into $\prod_{i=1}^r \bP(V_{\om_i}[[t]])$. This embedding gives rise to a family of
line bundles $\cO(\la)$ on $\Q$. The following theorem is proved in \cite{BF1,Kat}:
\[
H^0(\Q,\cO(\la))\simeq \gW_\la^*.
\]
We denote the projective coordinate ring of $\Q$ by $\gW^*=\bigoplus_{\la\in P_+} \gW_\la^*$.

\subsection{Integrable representations}
Let $\gh=\fg\T\bC[t,t^{-1}]\oplus\bC K\oplus\bC d$ be the affine Kac-Moody Lie algebra (here $K$ is central and $d$ satisfies
$[d,x\T t^i]=-ix\T t^i$).  The Cartan subalgebra $\fh^a$ of $\gh$ is equal to $\fh\T 1\oplus\bC K\oplus\bC d$.
Let $\delta\in (\fh^a)^*$ be the basic imaginary root. For an affine weight we define its finite part as the projection to $\fh^*$.

Let $\Lambda_0,\dots,\Lambda_d$ be the level one integrable weights (i.e. $\Lambda_i(K)=1$).
In what follows we use the fact that $d+1=|P/Q|$. Let $W^a$ be the affine Weyl group
(the Weyl group of $\gh$) and let $L(\Lambda_i)$ be the integrable level one representations of $\gh$.
Let
\begin{equation}
\mathbb{L}=\bigoplus_{i=0}^d L(\Lambda_i).
\end{equation}
For any integral weight $\la\in P$ there exists an element $w\in W^a$ and $0\le i\le d$ such that
the finite part of $w\Lambda_i$ is equal to $\la$. In other words, for any $\la\in P$ there exists extremal vector
$u_\la\in\mathbb{L}$ of finite weight $\la$. Following \cite{KL} we introduce the notation ($\la\in P_+$):
\begin{equation}
F_\la=\U(\fg\T\bC[t^{-1}])u_\la\subset \mathbb{L}.
\end{equation}
In particular, each $F_\la$ is $\fg\T\bC[t^{-1}]$ module and $F_\la\supset F_\mu$ if $\la<\mu$.
The following statement is proved in \cite{KL}:
\begin{equation}\label{grF}
\frac{F_\la}{\sum_{\mu>\la} F_\mu}\simeq \gW_\la.
\end{equation}
The isomorphism \eqref{grF} is the isomorphism of $\fg\T \bC[t^{-1}]$ modules, where $\gW_\la$ is made into
$\fg\T\bC[t^{-1}]$ module by the substitution $t\mapsto t^{-1}$.

Let $\sigma$ be Chevalley antiautomorphism on $\widehat{\fg}$ defined by
\begin{gather*}
\sigma(e_{\alpha} \otimes t^k)=f_{-\alpha}\otimes t^{-k},  \alpha \in \Delta;\\
\sigma(h \otimes t^k)=h\otimes t^{-k},  h \in \fh.
\end{gather*}
We consider the Shapovalov bilinear form $(\cdot,\cdot):L(\Lambda_i)\otimes L(\Lambda_i) \rightarrow \mathbb{C}$ with the property
$(x.u,v)=(u,\sigma(x).v)$ for all $x\in\gh$.

\begin{definition}
We define the increasing filtration on $L(\Lambda_i)$ by subspaces
\begin{equation}\label{Fdag}
F_\lambda^\dag=\bigcap_{\mu>\lambda^*} {\rm{ann}} F_{\mu},\ \la\in P_+
\end{equation}
where ${\rm{ann}} F_{\mu}$ is the orthogonal complement to $F_\mu$ with respect to the Shapovalov form.
\end{definition}

\begin{rem}\label{ula}
The notation for $F^\dag_\la$ is chosen in such a way that $u_\la\in F^\dag_\la$.
In fact, the Shapovalov form pairs nontrivially $u_\la$ and $u_{-\la}$. The vector $u_{-\la}$ is the lowest weight vector
of $\fg$-module, whose highest weight vector is $u_{-w_0\la}=u_{\la^*}$. We conclude that if $\mu>\la^*$, then
$u_\la\in {\rm{ann}} F_{\mu}$.
\end{rem}

Note that each $F_\lambda^\dag$ is a $\fg[t]$-module (since $F_\la$ is $\fg[t^{-1}]$-invariant).
\begin{prop}\label{Fdagsubquotient}
We have isomorphism of $\fg[t]$-modules
\[
F_{\lambda}^\dag/\sum_{\nu<\lambda}F_{\nu}^\dag\simeq \gW_{\lambda^*}^*.
\]
\end{prop}
\begin{proof}
It suffices to show that
\[
F_{\lambda}^\dag/\sum_{\nu<\lambda}F_{\nu}^\dag=(F_{\la^*}/\sum_{\mu>\la^*} F_\mu)^*.
\]
We have a pairing $(\cdot,\cdot):F_{\lambda^*}\otimes F_{\lambda}^{\dag}\rightarrow \mathbb{C}$ and $\sum_{\mu>\la^*}F_{\mu}$
is the annihilator
of $F_{\lambda}^{\dag}$, $\sum_{\nu<\lambda}F_{\nu}^\dag$ is the annihilator of $F_{\lambda^*}$.
\end{proof}

In what follows fro a filtration $F_\la$ we use the notation $F{<\lambda}=\sum_{\nu<\lambda}F_{\nu}$ and similarly for $F_{\le\la}$.

We have the following characterization of $\fg[t]$-modules $F_{\lambda}^{\dag}$ which can be immediately obtained from the duality.

\begin{prop}\label{Fdagcharacterization}
Let us fix $i$ such that $u_\la\in L(\Lambda_i)$. Then
$F_{\lambda}^{\dag}$ consists of vectors $v \in L(\Lambda_i)$ such that all (finite) $\fh$-weights of elements $\U(\fg[t])v$
are less than or equal to $\lambda$.
\end{prop}
\begin{proof}
Assume that for some $x \in \U(\fg[t])$ the weight $\mu$ of the element $xv$ is not less than or equal to $\lambda$.
Then (since $L(\Lambda_i)$ is of level one) for some element $y \in \U(\h[t])$ we have
\[yxv=u_\mu.\]
Thus $u_\mu\in F_{\lambda}^\dag$, which implies $\mu\le \la$ (see Remark \ref{ula}).
\end{proof}

\section{Generalities on $\Gamma$-graded vertex algebras}\label{gVOA}
\subsection{Definition and basic properties.}
In this section we recall the general concept of $\Gamma$-graded vertex algebras following \cite{R,DL}.
\begin{rem}
We only use the $\Gamma$-graded vertex algebras of weight lattices with $\Gamma=P/Q$ (weight lattice modulo root lattice).
For example, in the $E_8$-case $\Gamma$ is trivial and we obtain the usual vertex algebra.
\end{rem}

Let $\Gamma$ be a finite abelian group, $N$ be exponent of $\Gamma$, i. e. the smallest integer such that $N\Gamma=\{0\}$.
Let $\Delta$ be a $\mathbb{Z}$-bilinear symmetric map $\Gamma \times \Gamma \rightarrow \mathbb{Q}/\mathbb{Z}$
and $\nu:\Gamma \times \Gamma \rightarrow \mathbb{C}^*$ be a map satisfying
\[
\nu(g,h)\nu(h,g)=e^{-2\pi i \Delta(g,h)} \forall g,h \in \Gamma;\ ~\nu(g,g)=e^{-\pi i \Delta(g,g)}.
\]

For a $\Gamma$-graded vector space $V=\bigoplus_{g \in \Gamma}V_g$ the algebra ${\rm End}(V)$ is $\Gamma$-graded in the usual way.
An element $X(z)=\sum_{n \in \frac{1}{N}\mathbb{Z}}x_{(n)} z^{-1-n} \in {\rm End}(V)[[z^{\pm 1/N}]]$ is called a
{\it generalized vertex operator}
if for any $v \in V$ $x_{(n)}v=0$ for $n>>0$. The vertex operator $X(z)$ is called homogeneous of parity $g$ if all
$x_{(n)} \in {\rm End}(V)_g$
(i. e. $x_{(n)}V_h\subset V_{h+g}$ for any $h \in \Gamma$) and
for all $v \in V_h$ $x_{(n)} v=0$ unless $n \in \Delta(g,h)$. The parity of a homogeneous element $a$ is denoted by $p(a)$.

For an arbitrary rational $n$ denote by $i_{z,w}(z-w)^n$ the formal expansion of $(z-w)^n$ in the domain $|w|<|z|$:
\[
i_{z,w}(z-w)^n:=\sum_{j=0}^{\infty}(-1)^j\binom{n}{j}z^{n-j}w^j,\ \binom{n}{j}=\frac{n(n-1)\dots (n-j+1)}{j!}.
\]
Analogously denote by $i_{w,z}(z-w)^n$ the formal expansion of $(z-w)^n$ in domain the $|z|<|w|$:
\[
i_{w,z}(z-w)^n:=e^{in\pi}\sum_{j=0}^{\infty}(-1)^j\binom{n}{j}w^{n-j}z^j.
\]

\begin{definition}\label{locality}
Two homogeneous generalized vertex operators $X(z)$, $Y(z)$ of degrees $g$ and $h$ are mutually local with respect to $\Delta$ and $\nu$
if for large enough $n \in \Delta(g,h)$:
\[i_{z,w}(z-w)^n X(z)Y(w)=i_{w,z}(z-w)^n\nu(g,h)Y(w)X(z).\]
\end{definition}

\begin{definition}\label{normalorder}
For two homogeneous vertex operators $A(z)=\sum a_{(n)}z^{-1-n}$, $B(z)$ define their normally ordered product $:A(z)B(z):$ by
\[:A(z)B(z):=\sum_{n<0} a_{(n)}z^{-1-n}B(z)  +B(z)\sum_{n\ge 0} a_{(n)}z^{-1-n}.\]
\end{definition}

\begin{definition}\label{graded_vertex_algebra}

A $\Gamma$-graded vertex algebra is a vector space $V=\bigoplus_{g \in \Gamma}V_g$ with the following data:

$(i)$ a vacuum vector $|0\ket \in V_0$;

$(ii)$ a linear operator $T$ (translation operator);

$(iii)$  generalized vertex operators labeled by $a \in V$
\[Y(a,z)=\sum_{n \in \frac{1}{N}\mathbb{Z}}a_{(n)} z^{-1-n} \in {\rm End}(V)[[z^{1/N},z^{-1/N}]]\]

satisfying the following properties

$(0)$ the operation $Y(a,z)$ is linear in $a$ and for a homogeneous $a \in V_g$ has parity $g$;

$(1)$ (Translation) $[T,Y(a,z)]=\frac{\partial}{\partial z} Y(a,z)$ and $T|0\ket=0$;

$(2)$ (Vacuum) $Y(|0\ket,z)={\rm Id}_V$, $Y(a,z)|0\ket \in {\rm End}(V)[[z^{1/N}]$, $Y(a,0)|0\ket~=~a$;

$(3)$ (Locality) For all homogeneous $a,b \in V$ the operators $Y(a,z)$ and $Y(b,z)$ are mutually local with respect to $\Delta$
and $\nu$.

\end{definition}

We need the following property of $\Gamma$-graded vertex algebras (see \cite{R}, Theorem 4.1.25).

\begin{prop}\label{25}
For any homogenous elements $a,b,c \in V$, any $n \in {\Delta}(p(a),p(b))$,
$m \in {\Delta}(p(a),p(c))$, $k \in {\Delta}(p(b),p(c))$:
\begin{multline}\label{eq25}
\sum_{j=0}^{\infty} \binom{m}{j} (a_{(n+j)}b)_{(m+k-j)}c\\
=\sum_{j=0}^{\infty} (-1)^j \binom{n}{j} \big( a_{(n+m-j)}b_{(k+j)}-e^{n\pi i}\nu(p(a),p(b))b_{(n+k-j)}a_{(m+j)} \big)c.
\end{multline}
\end{prop}

We need the following corollary.

\begin{cor}\label{26}
For any $a \in V_0$, $b,c \in V$, $m \in \mathbb{Z}$, $k \in \Delta(p(b), p(c))$:
\begin{equation}\label{eq26}
[a_{(m)}, b_{(k)}]c=\sum_{j=0}^{\infty}\binom{m}{j}(a_{(j)}b)_{(m+k-j)}c.
\end{equation}
\end{cor}
\begin{proof}
This is a particular case of \eqref{eq25} for $n=0$.
\end{proof}

\begin{cor}\label{exactlocality}
Let $n \in {\Delta}(p(a),p(b))$ satisfies $a_{(n+j)}b=0$ for any integer $j\geq 0$. Then:
\[i_{z,w}(z-w)^nY(a,z)Y(b,w)=i_{w,z}(z-w)^n\nu(p(a),p(b))Y(b,w)Y(a,z).\]
\end{cor}
\begin{proof}
Due to equation \eqref{eq26} the difference between the coefficients of left and right hand sides of this formula
applied to an element $c \in V$ is equal to
\[\sum_{j=0}^{\infty} \binom{m}{j} (a_{(n+j)}b)_{(m+k-j)}c\]
which vanishes.
\end{proof}

\subsection{Lattice $\Gamma$-graded vertex algebras.}

Let $P$ be a lattice, i .e. a finitely generated free abelian group with a bilinear map $(\cdot,\cdot):P\times P \rightarrow \mathbb{Q}$.
Let $Q\subset P$ be an even integral sublattice of the same rank as $P$ such that $(Q,P)\subset \mathbb{Z}$.
Let $\Gamma=P/Q$,  $\Gamma=\{g_0, g_1, \dots, g_d\}$.

We denote by $p$ the natural projection $P \twoheadrightarrow \Gamma$.
For an element $g_i \in\Gamma$ we fix its representative $\chi_i \in P$.
Define the $\mathbb{Z}$-bilinear map $\Gamma \times \Gamma\rightarrow \mathbb{Q}/\mathbb{Z}$:
\[\Delta(g_i, g_j):=\mathbb{Z}-(\chi_i, \chi_j).\]
It is easy to see that such a form does not depend on a choice of representatives.
Next we define the following map $\Gamma\times \Gamma \rightarrow \mathbb{C}^*$:
\[\nu(g_i,g_j):=e^{i\pi(\chi_i,\chi_j)}\]
and the function $B:P \times P \rightarrow \mathbb{C}^*$:
\[
B(\lambda,\mu)=
e^{-i\pi (\lambda,\mu)}\nu(p(\lambda),p(\mu)).
\]
Note that $B(P\times P)\subset \{\pm 1\}$.
We fix a cocycle $\epsilon$ on $P$ with the following property:
\[
\epsilon(\omega_i, \omega_j)=
\begin{cases}
1,&~ {\text{ if }} ~i\leq j;\\
B(\omega_i, \omega_j),&~ {\text{ if }} ~i>j.
\end{cases}
\]
Such a $2$-cocycle exists (see \cite{S,DL}) and defines a new structure of an associative algebra on the space $\mathbb{C}[P]$:
\[e^{\lambda}e^{\mu}=\epsilon(\lambda,\mu)e^{\lambda+\mu}.\]
This structure is obtained from the projective representation of $P$ on its group algebra constructed via the cocycle $\epsilon$.

Our next task is to define the space of states of our VOA.
Let $\fh=\mathbb{C}\otimes_{\mathbb{Z}}P$. The form $(\cdot,\cdot)$ can be extended
to $\fh$ by linearity. Consider the vector space $\fh[t,t^{-1}]$
of formal Laurent polynomials in $t$ with coefficients in $\fh$. For $h \in \fh$, $k \in \mathbb{Z}$ we set
$ht^k=h \otimes t^k$. Let $\widehat{\fh}=\fh[t,t^{-1}]\oplus \mathbb{C}c$
be the Heisenberg algebra, i.e. the Lie algebra with the following bracket:
\[
[h_1t^k,h_2t^l]=k\delta_{k+l,0}(h_1,h_2)c
\]
and $c$ is a central element. Then the Fock (Verma) module of $\widehat{\fh}$ is the module
$U(\widehat{\fh})/\big(U(\widehat{\fh})\fh[t]+U(\widehat{\fh})(c-1)\big)$.

Consider the vector space $t^{-1}\fh[t^{-1}]$
of formal polynomials in $t^{-1}$ with coefficients in $\fh$ without zero term and let
$\mathbb{L}=S(t^{-1}\fh[t^{-1}])\otimes \mathbb{C}[P]$.
This space has a structure of an associative algebra.
Let $T$ be the following derivation of this algebra:
\[T(ht^{-n}\otimes 1)=n(ht^{-n-1}\otimes 1),~ T(1\otimes e^{\lambda})=\lambda t^{-1}\otimes e^{\lambda}, h \in \fh, \lambda \in P.\]
Let $|0\ket:=1\otimes e^0$ be the vacuum vector.
We identify $S(t^{-1}\fh[t^{-1}])$ with the Fock space.

For an element $h \in \fh$ we define the following generalized vertex operators
\begin{gather*}
Y(ht^{-1}\otimes 1,z)=\sum_{k\in\mathbb{Z}}ht^kz^{-1-k},\\
Y(ht^{-1-n}\otimes 1,z)=\frac{1}{n!}\frac{\partial^n}{\partial z^n}Y(ht^{-1},z).
\end{gather*}
Then define
\begin{equation}\label{VO}
Y(1 \otimes e^{\lambda},z)=\exp(\sum_{n>0}\frac{\lambda t^{-n}}{n}z^{n})\exp(\sum_{n<0}\frac{\lambda t^{-n}}{n}z^{n})\otimes e^{\lambda}z^{\lambda},
\end{equation}
where $z^{\lambda}e^{\mu}=e^{\mu}z^{(\lambda,\mu)}$.
Finally, for any $a \in t^{-1}\fh[t^{-1}]$, $b\in \mathbb{L}$
\[Y(ab,z)=:Y(a,z)Y(b,z):.\]

\subsection{Lie algebra action}
Let $\mathbb{L}$ be $\Gamma$-graded lattice vertex algebra of the positively defined lattice $P$.
Let $\{\alpha_i\}, \{\omega_i\}$ be dual bases of $\fh$. Consider the element
$\omega=\sum_{i=1}^r \alpha_i t^{-1}\cdot \omega_i t^{-1}\otimes e^0$ (the conformal vector).
Then the operator $\omega_{(1)}$ is diagonalizable and an element $\prod_{i=1}^n h_it^{-s_i}\otimes e^{\lambda}$ is an eigenvector
with the eigenvalue $\frac{(\lambda,\lambda)}{2}+\sum_{i=1}^n s_i$ (see, for example, \cite{R}, Chapter 4.3). This eigenvalue is called the
conformal weight.
Let $\mathbb{L}[i]$ be the space of elements of conformal weight $i$.
\begin{prop}
The operators $A_{(n)}$, $A\in \bL[1]$, $n \in \mathbb{Z}$ form an untwisted affine Lie algebra corresponding to the finite root system
$\{\alpha \in P| (\alpha, \alpha)=2\}$,
$e_{\alpha}t^n \mapsto (1\otimes e^{\alpha})_{(n)}$, $h_{\alpha}t^n \mapsto (\alpha t^{-1}\otimes e^0)_{(n)}$.
If $P$ is a weight lattice of the simple Lie algebra $\fg$, then for each $g \in P/Q$
$\mathbb{L}_g$ is a level-one integrable representation of $\widehat{\fg}$.
\end{prop}
\begin{proof}
The first claim can be easily seen from formula \eqref{26}. The second claim is contained in \cite{FK}.
\end{proof}

We close this section with an example.
\begin{example}\label{sl2ex}
Let $\fg=\msl_2$. Then $\Gamma=\bZ/2\bZ=\{e,g\}$. One has
\begin{gather*}
\Delta(e,e)=\Delta(e,g)=\Delta(g,e)=\bZ, \Delta(g,g)=-1/2+\bZ,\\
\nu(e,e)=\nu(e,g)=\nu(g,e)=1, \nu(g,g)=e^{i\pi/2}.
\end{gather*}
The space of states is given by $\bL=L(\Lambda_0)\oplus L(\Lambda_1)$.
For vectors $e^{\pm\omega}=1\T e^{\pm\omega}\in L(\Lambda_1)$ one gets
\begin{gather*}
Y(e^{\pm\om},z)|_{L(\Lambda_0)}=\sum_{n\in\bZ} e^{\pm\om}_{(n)}z^{-n-1},\\
Y(e^{\pm \om},z)|_{L(\Lambda_1)}=\sum_{n\in -1/2+\bZ} e^{\pm\om}_{(n)}z^{-n-1}.
\end{gather*}
Finally, the conformal weight of $e^{\pm\om}$ is equal to $1/4$.
\end{example}

\section{The strategy}\label{strategy}
Recall the space $\mathbb{L}=\bigoplus_{i=0}^d L(\Lambda_i)$, the direct sum of all level
one integrable $\gh$ modules. The space $\bL$ carries the structure of $P/Q$-graded vertex operator algebra.
\begin{rem}
If $\fg$ is of type $A_n$, then $d=n$. If $\fg$ is of type $E_8$, then $d=0$ (i.e. the only level one integrable
module is the basic one, $L(\Lambda_0)$).
\end{rem}

Recall that $\mathbb{L}$ is filtered by the (increasing)
filtration $F_\la^\dag$ and according to Proposition \ref{Fdagsubquotient} the associated graded
space is isomorphic to the direct sum of dual global Weyl modules $\gW^*=\bigoplus_{\la\in P_+} \gW_\la^*$.
This direct sum $\gW^*$ carries natural structure of commutative associative algebra, where the multiplication
is induced by the natural embeddings $\gW_{\la+\mu}\hk \gW_\la\T\gW_\mu$ (see \cite{Kat}).
Our goals are as follows:
\begin{itemize}
\item to describe the filtration $F_\la^\dag$ in terms of the VOA structure on $\bL$;
\item to describe the multiplication in $\gW^*$ in terms of the VOA $\bL$;
\item to endow the space $\gW^*$ with the structure of vertex operator algebra;
\item to derive (conjecturally, defining) relations in $\gW^*$ using the VOA structure.
\end{itemize}

Our strategy is as follows. We define an increasing filtration $G_\la\subset \bL$, $\la\in P_+$ using the action of vertex operators.
In short, the definition works as follows. We have natural projection from global to local Weyl modules
$\gW_\la\to W_\la$, which induce the embeddings $W^*_\la\subset \gW_\la^*$.
For each $i=1,\dots,r$ we fix a lift $\widetilde W_{\om^*_i}^*\subset F^\dag_{\om_i}$
(recall that $F^\dag_\la/F^\dag_{<\la}\simeq \gW_{\la^*}^*$). Then $G_\la$
is a linear span of the monomials of the form
\[
\tilde A^{(1)}_{(m_1)}\dots \tilde A^{(s)}_{(m_s)}|0\rangle,
\]
where $\tilde A^l$ are elements of $\widetilde W_{\om_i^*}^*$ and for each $i=1,\dots,r$ there are exactly $(\la,\al_i^\vee)$
vectors $\tilde A^l$ from $\widetilde W_{\om^*_i}^*$. We first prove that
\begin{itemize}
\item $A_{(n)}B\in G_{\la+\mu}^\dag$ for all $A\in G^\dag_\la$, $B\in G^\dag_\mu$, $n\in\frac{1}{N}\bZ$;
\item $A_{(n)}B\in F_{\la+\mu}^\dag$ for all $A\in F^\dag_\la$, $B\in F^\dag_\mu$, $n\in\frac{1}{N}\bZ$;
\item $G_\la\subset F_\la^\dag$ for all $\la\in P_+$.
\end{itemize}

\begin{example}
Let $\fg=\msl_2$. In the notation of Example \ref{sl2ex}
\[
G_{m\om}=\mathrm{span}\{e^{\pm\om}_{(n_1)}\dots e^{\pm\om}_{(n_m)}|0\ket, n_i\in\bZ\}.
\]
\end{example}

\begin{rem}
If $\fg$ is of type $A_n$, then $\sum_{\la<\om_i} F^\dag_\la=0$ for all $i$ (since $\{\la\in P_+, \la<\om_i\}=\emptyset$)
and $F^\dag_{\om_i}\subset L(\Lambda_i)$ for $i=1,\dots,n$.
In particular, $({\rm gr}F^\dag)_{\omega_i}=F^\dag_{\om_i}\simeq \gW_{\om^*_i}^*$.
However, this is not true in types $D$ and $E$. For example, in type $E_8$ all the spaces $F^\dag_{\om_i}$ belong to the same (the only)
integrable level one $\gh$ module.
\end{rem}

Recall the lifts $\widetilde W^*_{\om^*_i}\subset F^\dag_{\om_i}$. By Proposition \ref{globalweylconstruction}
we have an isomorphism of $\fg[t]$-modules
\[
({\rm gr}F^\dag)_{\omega_i}\simeq  W^*_{\om^*_i}[t^{-1}].
\]
In order to write down the multiplication in the algebra $\gW^*=\bigoplus_{\la\in P_+} \gW_\la^*$ we proceed as follows.
Let $\la=\om_{j_1}+\dots+\om_{j_s}$ and let us consider a collection of elements $A^k\in W_{\om^*_{j_k}}^*$ and their lifts
$\tilde A^k\in \widetilde W_{\om^*_{j_k}}^*$, $k=1,\dots,s$ and collection of nonnegative integers
$m_1,\dots,m_s$. We write down a formula for the product of the elements $A^k\T t^{-m_k}\in \mathbb{W}_{\om^*_{j_k}}^*$.
Consider the expression
\begin{equation}\label{multform}
\eM(\tilde A^1,\dots,\tilde A^s)=\prod_{1\le k<l\le s}i_{z_k,z_l} (z_k-z_l)^{-(\om^*_{i_k},\om^*_{i_l})}
\prod_{k=1}^s Y(\tilde A^k,z_k)|0\rangle.
\end{equation}
Here we use the notation $\prod_{i=1}^sX_i=X_1\dots X_s$ for possible noncommuting $X_i$.
We denote by $\eM(\tilde A^1,\dots,\tilde A^s)_{m_1,\dots,m_s}$ the coefficient of $z_1^{-m_1}\dots z_s^{-m_s}$ in
$\eM(\tilde A^1,\dots,\tilde A^s)$. We show that the multiplication rule
\begin{equation}\label{mult}
A^1\T t^{-m_1}\dots A^s\T t^{-m_s} = \eM(\tilde A^1,\dots,\tilde A^s)_{m_1,\dots,m_s}
\end{equation}
induces the map
\[
\varphi: \gW_{\om^*_{j_1}}^*\T\dots\T \gW_{\om^*_{j_s}}^*\to G_\la/(G_\la\cap F^\dag_{<\la}).
\]
We then prove that
\begin{itemize}
\item $\varphi$ does not depend on the choice of lifts $\widetilde W_{\om^*_{j_k}}^*$;
\item $\varphi$ is commutative;
\item $\varphi$ is a homomorphism of $\fg[t]$-modules.
\end{itemize}

From this we derive that
\begin{itemize}
\item $G_\la=F^\dag_\la$ for all $\la\in P_+$;
\item the multiplication \eqref{mult} coincides with the multiplication in $\gW^*$.
\end{itemize}

\begin{rem}\label{VOAW}
The $P/Q$-graded vertex operator algebra structure on $\gW^*$ comes from the property
$A_{(n)}B\in G_{\la+\mu}$ for all $A\in G_\la$, $B\in G_\mu$ and $n\in \frac{1}{N}\bZ$.
\end{rem}

Now let $A\in W^*_{\om^*_i}$, $B\in W^*_{\om_j^*}$ and let us consider $A$ and $B$ as elements of $P/Q$-graded VOA $\gW^*$
(via the embedding $W^*_{\om^*_i}\subset \gW^*_{\om^*_i}$ and Remark \ref{VOAW}). Let $m_{i,j}=-(\om_i^*,\om_j^*)$.
We show that $(A_{(m_{i,j}-s)}B)_{(-1-r)}|0\ket$ is equal to the coefficient of $z^r$ in
\begin{equation*}
\Bigl(\partial^{s-1}_z i_{z,w}(z-w)^{m_{i,j}} Y(A,z) Y(B,w)\Bigr) |_{z=w}|0\ket.
\end{equation*}
This observation produces a set of quadratic relations for the Pl\"ucker-Drinfeld embedding of the semi-infinite flag varieties.
We conjecture that this is a defining set of relations and prove the conjecture in type $A$.

\section{VOA filtration and multiplication formula}\label{Main}
\subsection{VOA filtration}
Let us introduce the increasing filtration on $\bL$ which should be understood as the $\Gamma$-graded analogue of the so-called
standard filtration (see \cite{A,Li1,Li2}). The filtration spaces $G_\la\subset \bL$ are labeled  by the dominant integer weights
$\la\in P_+$. The filtration is defined inductively as follows:
\begin{equation}\label{Gdef}
G_0=\bC|0\rangle,\
G_{\la}=\sum_{\mu<\la}G_\mu + \sum_{i:\ \la-\om_i\in P_+}
{\mathrm{span}}\{B_{(k)}G_{\la-\omega_i}, B\in \widetilde W^*_{\om^*_i}, k\in \frac{1}{N}\bZ\}.
\end{equation}
\begin{rem}
Let $\fg$ be of type $A_n$. Then $W^*_{\om^*_i}\simeq V_{\om^*_i}^*=F^\dag_{\om_i}$ is the top of the integrable $\gh$-module $L(\Lambda_i)$.
\end{rem}

For a subspace $U\subset \bL$ we use the notation $U_{(n)}=\{u_{(n)}: u\in U\}$.
\begin{prop}
For any $n \in \frac{1}{N}\mathbb{Z}$: $(G_{\lambda})_{(n)}G_{\mu}\subset G_{\lambda+\mu}$.
\end{prop}
\begin{proof}
We prove Proposition by induction on $\lambda$. The basis of induction is the trivial case $\lambda=0$.
Assume first that $\lambda=\omega_i$ is the fundamental weight. Then for $A \in \widetilde{W}^*_{\omega^*_i}$ we need to prove the following:
\[(A_{(k)}|0\ket)_{(n)}G_{\mu}\subset G_{\mu + \omega_i}.\]
However the operators $(A_{(k)}|0\ket)_{(n)}$ are scalar multiples of operators $A_{(k+n+1)}$. Therefore Proposition holds by
definition of $G_{\mu}$.

For general $\lambda$ take an element $A_{(k)}B$, $A \in \widetilde{W}^*_{\omega^*_i}$, $B \in G_{\lambda-\omega_i}$.
Then using equation \eqref{25} we obtain that for any $C$:
\[(A_{(k)}B)_{(n)}C\in \sum_{i \in \mathbb{Z}} \mathbb{C}A_{(k+i)}(B_{(n-i)}C)+\sum_{i \in \mathbb{Z}} \mathbb{C}B_{(n-i)}(A_{(k+i)}C).\]
So the proof is completed by induction.
\end{proof}



\begin{prop}\label{Fdagmult}
For any $n \in \frac{1}{N}\bZ$:
\[(F_{\lambda}^{\dag})_{(n)}F_{\mu}^{\dag}\subset F_{\lambda+\mu}^{\dag}.\]
\end{prop}
\begin{proof}
Using formula \eqref{26} for any $x \in \fg$, $r \in \mathbb{Z}_+$ we have:
\[xt^k\big((F_{\lambda}^{\dag})_{(n)}F_{\mu}^{\dag}\big)\subset \sum_{i \in \mathbb{Z}}(F_{\lambda}^{\dag})_{(n+i)}F_{\mu}^{\dag}.\]
In other words $\sum_{i \in \mathbb{Z}}(F_{\lambda}^{\dag})_{(n+i)}F_{\mu}^{\dag}$ is $\fg[t]$-submodule. However for any scalar current
$h \in \fg$ over Cartan subalgebra:
\[h(A_{(n)}B)=(h.A)_{(n)}B+A_{(n)}(h.B).\]
Therefore the set of weights of $\sum_{i \in \mathbb{Z}}(F_{\lambda}^{\dag})_{(n+i)}F_{\mu}^{\dag}$ is contained in the
sum of the sets of weights of
$F_{\lambda}^{\dag}$ and $F_{\mu}^{\dag}$. In particular, all weights of this $\fg[t]$-submodule are less than or equal to
$\lambda$. Thus this submodule
is contained in $F_{\lambda+\mu}^{\dag}$.
\end{proof}

\begin{lem}
For all $\la\in P_+$ one has $G_\la\subset F^\dag_\la$.
\end{lem}
\begin{proof}
We prove the claim by induction with respect to the partial order \eqref{order}. The $\la=0$ case is obvious.
Now assume our lemma is proved for all weights less than or equal to $\la$. In order to prove the inclusion
$G_\la\subset F^\dag_\la$ it suffices to show that all the weights of the $\U(\fg[t])$-span of $G_\la$ are less than or
equal to $\la$. Recall definition \eqref{Gdef}.
We first note that all the weights of $G_\la$ are less than or equal to $\la$ (see Proposition \ref{Fdagcharacterization}).
In fact, let $B\in \widetilde W^*_{\om^*_i}$, $C\in G_{\la-\omega_i}$ (we assume $\la-\om_i\in P_+$).
Then for any $h\in\fh$ and $k\in \frac{1}{N}\bZ$ one has
\[
h(B_{(k)}C)=B_{(k)}(h.C)+ (h.B)_{(k)}C,
\]
which gives the statement by the induction assumption.

Now assume (by induction) that  for all weights $\mu < \lambda$ the $\U(\fg[t])$-span
of $G_{\mu}$ is
contained in $G_{\mu}+F_{<\mu}$. Then for $B\in \widetilde W^*_{\om^*_i}$, $C\in G_{\la-\omega_i}$
and $x\in\fg$, $m\ge 0$ one has (see Corollary \ref{26})
\[
xt^m(B_{(k)}C)=B_{(k)}(xt^m.C)+\sum_{j=0}^m \binom{m}{j} (xt^j.B)_{(m+k-j)}C,
\]
which belongs to $G_\la+F^\dag_{<\la}$. Therefore $G_\lambda\subset F_{\lambda}^\dag$ by Proposition \ref{Fdagcharacterization}.
This completes the proof of Lemma
\end{proof}

\subsection{Multiplication}
Let $\la=\sum_{k=1}^s \om_{j_k}\in P_+$.
Let $\tilde A^k\in \widetilde W_{\om^*_{j_k}}^*$, $k=1,\dots,s$ be a collection of elements in the lifts of the local Weyl modules.
For $i,j=1,\dots,r$ let $m_{i,j}=-(\om_i^*,\om_j^*)$.
We consider the expression
\begin{equation}\label{multiplication}
\eM(\tilde A^1,\dots,\tilde A^s)=
\prod_{1\le k<l\le s} i_{z_k,z_l}(z_k-z_l)^{m_{i_k,i_l}}\prod_{k=1}^s Y(\tilde A^k,z_k)|0\rangle
\end{equation}
(see e.g. \cite{DL}, Chapter 7, where the matrix coefficients of this expression are considered).
Let
\begin{equation}\label{Mcoef}
\eM(\tilde A^1,\dots,\tilde A^s)=\sum_{m_1,\dots,m_s} \eM(\tilde A^1,\dots,\tilde A^s)_{m_1,\dots,m_s} z_1^{m_1}\dots z_s^{m_s}.
\end{equation}

\begin{lem}
$\eM(\tilde A^1,\dots,\tilde A^s)_{m_1,\dots,m_s}\in G_\la$
for all $m_1,\dots,m_s$.
\end{lem}
\begin{proof}
Follows from the definition of $G_\la$.
\end{proof}

\begin{lem}\label{commute}
$\eM(e^{\om^*_{j_1}},\dots,e^{\om^*_{j_s}})$ is invariant under the permutation of indices $j_k$.
\end{lem}
\begin{proof}
Corollary \ref{exactlocality} says that
\begin{multline*}
i_{z_a,z_{a+1}}(z_a-z_{a+1})^{m_{j_a,j_{a+1}}}Y(e^{\om^*_{j_a}},z_a)Y(e^{\om^*_{j_{a+1}}},z_{a+1})=\\
\nu(p(\om^*_{j_a}),p(\om^*_{j_{a+1}}))i_{z_{a+1},z_a}(z_a-z_{a+1})^{m_{j_a,j_{a+1}}}Y(e^{\om^*_{j_{a+1}}},z_{a+1})Y(e^{\om^*_{j_a}},z_a).
\end{multline*}
Now it suffices to note that
\[
\nu(p(\om^*_{j_a}),p(\om^*_{j_{a+1}}))i_{z_{a+1},z_a}(z_a-z_{a+1})^{m_{j_a,j_{a+1}}}=
i_{z_{a+1},z_a}(z_{a+1}-z_a)^{m_{j_a,j_{a+1}}}.
\]
\end{proof}

Recall the definition \ref{Mcoef} of the elements $\eM(\tilde A^1,\dots,\tilde A^s)_{m_1,\dots,m_s}$.
\begin{lem}\label{positive}
$\eM(e^{\omega^*_{j_1}},\dots,e^{\omega^*_{j_s}})\in \bL[[z_1,\dots,z_s]]$ and
\begin{equation}\label{freeterm}
\eM(e^{\omega^*_{j_1}},\dots,e^{\omega^*_{j_s}})_{0,\dots,0}=\pm e^\la.
\end{equation}
\end{lem}
\begin{proof}
To prove the first claim, we start with $z_s$. Note that $Y(e^{\omega^*_{j_s}},z_s)|0\rangle\in\bL[[z_s]]$ and
\begin{equation}\label{prefactor}
\prod_{1\le k<l\le s} i_{z_k,z_l}(z_k-z_l)^{m_{j_k,j_l}}
\end{equation}
is a Taylor series in $z_s$. Now Lemma \ref{commute} implies the first claim.
The equality \eqref{freeterm} can be derived from the explicit formula \eqref{VO} for the action of vertex operators 
$Y(e^{\om^*_j},z)$.
\end{proof}

\subsection{The map}
Recall the identification $\gW^*_{\om^*_i}\simeq W^*_{\om^*_i}[t^{-1}]$. For an element $A\in W^*_{\om^*_i}$ let
$\tilde A\in \widetilde W^*_{\om^*_i}$ be
its lift in $F^\dag_{\om_i}$.  For a dominant integral weight $\la=\sum_{k=1}^s \om_{j_k}$ we define a map
\begin{gather}
\varphi: \gW^*_{\om^*_{j_1}}\T\dots \T \gW^*_{\om^*_{j_s}}\to G_\la/(G_\la\cap F^\dag_{<\la}),\\
\tilde A^1t^{-m_1}\T\dots\T \tilde A^st^{-m_s}\mapsto [\eM(\tilde A^1,\dots,\tilde A^s)_{m_1,\dots,m_s}]
\end{gather}
(the square brackets denote the class of an element of $G_\la$ in $G_\la/(G_\la\cap F^\dag_{<\la})$).

\begin{lem}\label{indep}
The map $\varphi$ does not depend on the lifts $\widetilde W^*_{\om^*_i}\subset F^\dag_{\om_i}$.
\end{lem}
\begin{proof}
For $A^i\in W^*_{\om^*_i}$ let us choose two lifts ${\tilde A}^i$ and ${\tilde{\tilde A}^i}$. Then
${\tilde A}^i-{\tilde{\tilde A}^i}\in F^\dag_{<\om_i}$ and the claim holds because of
Proposition \ref{Fdagmult}.
\end{proof}


\begin{lem}\label{currentaction}
Let us fix $A^l\in W^*_{\om^*_{j_l}}$, $l=1,\dots,s$ and $xt^m\in\fg[t]$ ($x\in\fg$, $m\ge 0$). Then
\begin{multline}
xt^m.\eM(\tilde A^1,\dots,\tilde A^s)_{m_1,\dots,m_s}=\\ \sum_{k=1}^s\sum_{p=0}^m \binom{m}{p}
\eM(\tilde A^1,\dots, xt^p.\tilde A^k, \dots,\tilde A^s)_{m_1,\dots,m_k-m,\dots,m_s}.
\end{multline}
\end{lem}
\begin{proof}
Let us compute the left hand side. Since $xt^m|0\rangle=0$, one has
\begin{multline*}
xt^m.\prod_{1\le k<l\le s} i_{z_k,z_l}(z_k-z_l)^{m_{j_k,j_l}}\prod_{k=1}^s Y(A^k,z_k)|0\rangle\\
=\prod_{1\le k<l\le s} i_{z_k,z_l}(z_k-z_l)^{m_{j_k,j_l}}\sum_{l=1}^s Y(A^1,z_1)\dots [xt^m, Y(A^l,z_l)]\dots Y(A^s,z_s)|0\rangle.
\end{multline*}
Now by \eqref{26} one has
\[
[xt^m, Y(A^l,z_l)]=\sum_{p=0}^m \binom{m}{p} Y(xt^p.A^l,z_l)z_l^{m-p}.
\]
\end{proof}

\begin{prop}\label{negative}
Fix a collection of integers $m_1,\dots,m_s$ and assume that $m_k<0$ for some $1\le k\le s$. Then
\[
\eM(\tilde A^1,\dots,\tilde A^s)_{m_1,\dots,m_s}\in F^\dag_{<\la}.
\]
\end{prop}
\begin{proof}
Thanks to Proposition \ref{Fdagcharacterization} it suffices to show that the weights of the $\U(\fg[t])$-span of
$\eM(\tilde A^1,\dots,\tilde A^s)_{m_1,\dots,m_s}$ are smaller than $\la$.
Note that this span lies in the linear span of elements
$\eM(\tilde B^1,\dots,\tilde B^s)_{a_1,\dots,a_s}$, $B^l \in \widetilde{W}_{\omega_{j_l}^*}^*$
and the weights of these elements are smaller than to equal $\leq \lambda$.
We note that if the weight of $\eM(\tilde B^1,\dots,\tilde B^s)_{m_1,\dots,m_s}$, $B^l\in W^*_{\om^*_{j_l}}$ is equal to $\lambda$
then the weight of $B^l$ is $\om^*_{j_l}$. Hence all the weight $\la$ vectors in the  $\U(\fg[t])$-span of
the elements $\eM(\tilde A^1,\dots,\tilde A^s)_{m_1,\dots,m_s}$ are of the form
\[
\eM(e^{\om^*_{j_1}},\dots, e^{\om^*_{j_s}})_{a_1,\dots,a_s}.
\]
However, Lemma \ref{currentaction} implies that $a_k<0$ (since $m\ge 0$ and $m_k<0$).
Now Lemma \ref{commute} and Lemma \ref{positive} imply our Proposition.
\end{proof}

\begin{thm}
The map $\varphi$ is a homomorphism of $\fg[t]$ modules.
\end{thm}
\begin{proof}
Lemma \ref{currentaction} and Proposition \ref{negative} imply
\begin{multline}
xt^m.[\eM(\tilde A^1,\dots,\tilde A^s)_{m_1,\dots,m_s}]=\\ \sum_{k=1}^s\sum_{p=0}^m \binom{m}{p}
[\eM(\tilde A^1,\dots, xt^p.\tilde A^k, \dots,\tilde A^s)_{m_1,\dots,m_k-m,\dots,m_s}],
\end{multline}
where the terms in the right hand side with $m_k<m$ vanish.
Now we derive the theorem from formula \eqref{globalization} and Proposition \ref{globalweylconstruction}.
\end{proof}

\begin{cor}\label{image}
The image of $\varphi$ surjects onto $\gW^*_{\la^*}$.
\end{cor}
\begin{proof}
We have a surjection $S$ of $\fg[t]$ modules ($\la^*=\sum_{k=1}^s \om_{j_k}^*$):
\[
S: \gW^*_{\om_{j_1}^*}\T\dots \T \gW^*_{\om_{j_s}^*}\mapsto \gW_{\la^*}^*.
\]
We claim that $S(\ker\varphi)=0$. Indeed, assume there exists $y\in \ker\varphi$ such that $S(y)\ne 0$.
Since $\gW^*_{\la^*}$ is cocyclic $\fg[t]$ module, there exists $\tau\in \U(\fg[t])$ such that $\tau.S(y)$ is the cocylcic
vector, which is mapped to a non-zero vector thanks to  \eqref{freeterm}, giving a contradiction.
We thus obtain an induced surjective map
\[
\gW^*_{\om_{j_1}^*}\T\dots \T \gW^*_{\om_{j_s}^*}/\ker\varphi\to \gW^*_{\la^*}.
\]
\end{proof}

\begin{cor}\label{GF}
$G_\la=F^\dag_\la$ for $\la\in P_+$.
\end{cor}
\begin{proof}
Since the image of $\varphi$ is contained in $G_\la/(G_\la\cap F^\dag_{<\la})$, Corollary \ref{image} implies the surjection
$G_\la/(G_\la\cap F^\dag_{<\la})\to \gW^*_{\la^*}$. However, $G_\la$ is contained in $F_\la^\dag$ and
$F_\la^\dag/F^\dag_{<\la}\simeq \gW^*_{\la^*}$.
\end{proof}

We thus obtain the following theorem. Recall the canonical embedding $W_\la^*\subset \gW_\la^*$.

\begin{thm}
The homogeneous coordinate ring $\gW^*=\bigoplus_{\la\in P_+} \gW_\la^*$ has a structure of $P/Q$-graded vertex operator
algebra generated by the fields $Y(A,z)$, $A\in W^*_{\om_i}$, $i=1,\dots,r$.
\end{thm}
\begin{proof}
We know (see Corollary \ref{GF}) that ${\rm gr} G_\bullet={\rm gr} F^\dag_\bullet\simeq \gW^*$. Since
$(F^\dag_\la)_{(n)}F^\dag_\mu\subset F^\dag_{\la+\mu}$,
the associated graded space ${\rm gr} F^\dag_\bullet$ carries the $P/Q$-graded vertex operator algebra structure.
Finally, the algebra ${\rm gr} G_\bullet$ is generated by the fields $Y(A,z)$, $A\in W^*_{\om_i}$, $i=1,\dots,r$
by definition. This completes the proof.
\end{proof}

\begin{rem}
Assume that $\fg$ is of type $E_8$. Then $\Gamma$ is a trivial group so we deal with the usual vertex algebra.
Let $\lambda=\sum_{i=1}^{s}\omega_{j_{i}}$.
For any $\tilde A^{l} \in \widetilde{W}_{\omega_{j_l}}$
we have $\tilde A^l_{(n)}F_{\lambda}^\dag\subset F_{<\lambda+\omega_i}^\dag$, $n \geq 0$.
Therefore for $B=\tilde A^{1}_{(-1-m_1)}\dots \tilde A^s_{(-1-m_s)}|0\ket$
one has $B_{(n)}F_{\mu}^\dag \subset F_{<\mu+\lambda}^\dag$, $n \geq 0$ (this is true only for $\Gamma=\{0\}$).
Thus in the described above vertex algebra on the space $\bigoplus_{\lambda \in P_{+}}\mathbb{W}_{\lambda}^*$ we have:
\[B_{(n)}C=0, n \geq 0, B, C \in \mathbb{L}.\]
Therefore this vertex algebra is holomorphic (aka commutative). Hence the bilinear operation
$(B,C)\mapsto B_{(-1)}C$ is
a commutative and associative, inducing another structure of commutative and associative algebra on the space
$\bigoplus_{\lambda \in P_{+}}\mathbb{W}_{\lambda}^*$. However in this algebra all the
elements $P \otimes e^{\lambda}$ are nilpotent, if $\lambda \neq 0$.
\end{rem}

\section{Semi-infinite Pl\"ucker relations via vertex algebras}\label{VOAPl}
In this section we use the vertex operator realization of the homogeneous coordinate ring of the
semi-infinite flag variety in order to derive semi-infinite Pl\"ucker relations.

Recall $m_{i_1,i_2}=-(\om^*_{i_1},\om^*_{i_2})$ and let us fix two vectors $A\in W^*_{\om^*_{i_1}}$, $B\in  W^*_{\om^*_{i_2}}$.

\begin{lem}
The following equality holds in the $P/Q$-graded VOA $\gW^*$:
\begin{multline}\label{expl}
i_{z,w}(z-w)^{m_{i_1,i_2}} Y(A,z) Y(B,w)|0\ket\\
=\sum_{p,q\ge 0} z^pw^q \sum_{j\ge 0} (-1)^j\binom{{m_{i_1,i_2}}}{j} A_{(m_{i_1,i_2}-1-p-j)}B_{(-1-q+j)}|0\ket.
\end{multline}
\end{lem}
\begin{proof}
A direct computation shows that the left hand side of \eqref{expl} considered as operators in $\bL$
is given by the right hand side with the only change $p,q\in\bZ$ (instead of $p,q\in\bZ_{\ge 0}$).
Now applying Proposition \ref{negative} we derive our Lemma.
\end{proof}

\begin{rem}
Recall that the conformal weight of $e^\lambda\in\bL$ (the cocylic vector of $\gW^*_{\lambda}$) is equal to $(\la,\la)/2$.
If $A=e^{\om^*_{i_1}}$, $B=e^{\om^*_{i_2}}$, then $A_{m_{i_1,i_2}-1}B$ is the cocylic vector if
$\gW^*_{\om^*_{i_1}+\om^*_{i_2}}$, which is equal to $e^{\om^*_{i_1}+\om^*_{i_2}}$
up to a sign (this can be seen from the explicit formula \eqref{VO} for vertex operators). If $s\le 0$, then
for any $\tilde A\in W^*_{\om^*_{i_1}}$, $\tilde B\in  W^*_{\om^*_{i_2}}$ vector  $\tilde A_{(m_{i_1,i_2}-s)}\tilde B$
belongs to $G_{<(\om_{i_1}+\om_{i_2})}$.
\end{rem}

We prove the following theorem.
\begin{prop}\label{VertexMult}
For any $s>0$, $r\ge 0$ the vector
$$(A_{(m_{i_1,i_2}-s)}B)_{(-1-r)}|0\rangle$$
is equal to the coefficient of $z^r$ in
\begin{equation}\label{partial}
\Bigl(i_{z,w}(z-w)^{m_{i_1,i_2}} \partial^{s-1}_z Y(A,z) Y(B,w)\Bigr) |_{z=w}|0\ket
\end{equation}
in the $P/Q$-graded VOA $\gW^*$.
\end{prop}
\begin{proof}
We use formula \eqref{eq25}. Let $c$ be the vacuum vector $|0\ket$, $m=0$, $n=m_{i_1,i_2}-s$, $k=-r-1$, $a=A$, $b=B$.
Then the left hand side sum of \eqref{eq25} reduces to a single $j=0$ term, which is equal to $(A_{(m_{i_1,i_2}-s)}B)_{(-1-r)}|0\rangle$.
In the right hand side all the terms $B_{(n+k-j)}A_{(m+j)}|0\rangle$ vanish since $m+j\ge 0$. Hence we are left with the equality
\[
(A_{(m_{i_1,i_2}-s)}B)_{(-1-r)}|0\rangle=\sum_{j\ge 0} (-1)^j \binom{{m_{i_1,i_2}}-s}{j}A_{({m_{i_1,i_2}}-s-j)}B_{(-1-r+j)}|0\rangle.
\]
Now formula \eqref{expl} implies that the coefficient in front of $z^k$ in \eqref{partial} is equal to
\begin{equation}
\sum_{l\ge 0}(-1)^l\binom{m_{i_1,i_2}}{l}\sum_{\substack{p\ge s-1, q\ge 0\\ p+q=r+s-1}}
p(p-1)\dots(p-s+2)A_{(m_{i_1,i_2}-1-p-j)}B_{(-1-q+l)}|0\rangle
\end{equation}
(we note that $p,q$ are nonnegative here thanks to
Proposition \ref{negative} which says that no coefficients in front of negative powers of $z,w$ survive in $\gW^*$).
So we are left to show that
\begin{multline*}
\sum_{j\ge 0} (-1)^j \binom{{m_{i_1,i_2}}-s}{j}A_{({m_{i_1,i_2}}-s-j)}B_{(-1-r+j)}\\
=\sum_{l\ge 0}(-1)^l\binom{{m_{i_1,i_2}}}{l}\sum_{\substack{p\ge s-1, q\ge 0\\ p+q=r+s-1}} p(p-1)\dots(p-s+2)A_{({m_{i_1,i_2}}-1-p-j)}
B_{(-1-q+l)}.
\end{multline*}
Combining the terms on the right, the equality is implied by
\begin{equation}\label{binom}
(-1)^j\binom{{m_{i_1,i_2}}-s}{j}=\sum_{l=0}^j (-1)^l\binom{{m_{i_1,i_2}}}{l}(j-l+s-1)(j-l+s-2)\dots (j-l+1).
\end{equation}
The proof of \eqref{binom} is standard: the generating function over $j$ of the left hand side terms is equal to 
$(1-x)^{{m_{i_1,i_2}}-s}$, while on the right the generating function is given by $(1-x)^{m_{i_1,i_2}}(1-x)^{-s}$.
\end{proof}

Given an element $A\in W^*_{\om^*_i}$ we denote by $A(z)=\sum_{r\ge 0} z^r (A\T t^{-r})\in \gW^*_{\om^*_i}[[z]]$
(we use the identification
$\gW^*_{\om^*_i}\simeq  W^*_{\om^*_i}[t^{-1}]$ from Proposition \ref{globalweylconstruction}).
Recall the isomorphism of vector spaces $\gW^*\simeq {\rm gr} G_\bullet$.
We rephrase Theorem \ref{VertexMult} in terms of multiplication in algebra $\gW^*$.

\begin{cor}\label{AB}
For any $s>0$ and $A\in W^*_{\om^*_{i_1}}$, $B\in  W^*_{\om^*_{i_2}}$ one has
\[
\frac{\partial^{s-1}A(z)}{\partial z^{s-1}}B(z)=\sum_{r\ge 0} z^r (A_{(m_{i_1,i_2}-s)}B)_{(-1-r)}|0\rangle,
\]
where the multiplication in the left hand side is understood inside $\gW^*$ and in the right hand side
the VOA structure on $\gW^*$ is used.
\end{cor}
\begin{proof}
Follows from the multiplication formula \eqref{multiplication}.
\end{proof}

We apply  Corollary \ref{AB} to the question of finding relations in the algebra $\gW^*$ (the so called semi-infinite Pl\"ucker relations).
\begin{thm}\label{descr}
Assume that for some positive $s$ one has
$$\sum_l \tilde A^l_{({m_{i_1,i_2}}-s)}\tilde B^l\in G_{<(\om_{i_1}+\om_{i_2})}$$
for some $A^l\in W^*_{\om_{i_1}^*}$, $B^l\in \tilde W^*_{\om_{i_2}^*}$. Then
all the coefficients of the series
$\sum_l (\partial^{s-1}A^l(z))B^l(z)$ vanish as elements of $\gW^*$.
\end{thm}
\begin{proof}
Follows from Corollary \ref{AB}.
\end{proof}
\begin{conj}
The relations on $\gW^*$ from Theorem \ref{descr} are defining.
\end{conj}
We prove this conjecture in type $A$ in the section below.

\section{Type A}\label{typeA}
In this section we consider the case $\fg=\msl_{r+1}$. The whole picture simplifies a lot because
the local Weyl modules $W_{\om_i}$ are irreducible as $\fg$ modules (i.e.  $W_{\om_i}\simeq V_{\om_i}$).

Let $\Lambda_1,\dots,\Lambda_r$ be the level one integrable weights of $\gh$ such that the restriction of $\Lambda_i$
to the Cartan subalgebra of $\fg$ is equal to $\omega_i$. Let $v_{\Lambda_i}\in L(\Lambda_i)$ be the highest weight vector.
\begin{lem}
For each $i=1,\dots,r$ the lift $\widetilde W^*_{\om_i}$ is unique. One has
\[
\widetilde W^*_{\om_i}=\U(\fg)v_{\Lambda_i}\subset L(\Lambda_i).
\]
\end{lem}
\begin{proof}
In type $A$ the inequality $\la\le \om_i$ implies $\la=\om_i$ for $\la\in P_+$.
\end{proof}

\begin{cor}
The space $F^\dag_{\om_i}=G_{\om_i}$ is isomorphic to $\gW^*_{\om_i^*}$ and is explicitly given by
\[
{\rm{span}}\{ A_{(n)}|0\ket, A\in W^*_{\om_i^*}, n\in\bZ\}.
\]
\end{cor}

Let us now explicitly describe the relations for the algebra $\gW^*$.
In \cite{FM2} for any two strictly increasing columns $I$ and $J$ such that length of $I$ is no smaller than the length of
$J$  we defined the number $k(I,J)$. The value $k(I,J)$ measures how far the two-column Young tableaux $Y(I,J)$ corresponding to
the pair $I,J$ is from being semi-standard. In particular, $k(I,J)=0$ if and only if the corresponding tableaux is semi-standard.
Let us recall the formal definition.

Let $I,J\subset \{1,\dots,r+1\}$, $I=(i_1<\dots <i_{l(I)})$, $J=(j_1<\dots <j_{l(J)})$. We assume that either $l(I)>l(J)$, or
$l(I)=l(J)$ and there exists $1\le s\le l(I)$ such that
\[
i_{l(I)}=j_{l(I)}, \dots, i_{s+1}=j_{s+1}, i_s<j_s.
\]
We construct a set $P=(p_1,\dots,p_M)\subset I\sqcup J$ step by step decreasing the index. We start with
fixing $p_M=i_{l(I)}$. Now at each step
$p_k$ is identified with an element $i_a\in I$ or $j_b\in J$. We proceed via the following rule:
\begin{itemize}
\item if $p_k=i_a$ and $i_a>j_a$, then $p_{k-1}=j_a$; otherwise $p_{k-1}=i_{a-1}$,
\item if $p_k=j_a$ and $i_a<j_a$, then $p_{k-1}=i_a$; otherwise $p_{k-1}=j_{a-1}$.
\end{itemize}

\begin{dfn}
$k(I,J)=|P(I,J)|-l(I)$.
\end{dfn}

Let $Y(I,J)$ be the two-column Young diagram whose left (longer) column is $I$ and the right column is $J$.

\begin{lem}
The following properties hold by construction:
\begin{itemize}
\item $Y(I,J)$ is semi-standard if and only if $k(I,J)=0$,
\item $p_1<p_2<\dots <p_M$,
\item $0\le k(I,J)\le \min(l(J),r+1-l(I))$.
\end{itemize}
\end{lem}
\begin{proof}
To prove the first claim we note that $k(I,J)=0$ is equivalent to $P(I,J)=I$, which means that $i_a\le j_a$ for all $a$.
The second claim immediately follows from the definition. Since all the numbers in $P$ are distinct and belong to $\{1,\dots,r+1\}$,
we conclude $k(I,J)+l(I)\le r+1$. The inequality $k(I,J)\le l(J)$ is obvious, since it is equivalent to $|P(I,J)|\le l(I)+l(J)$.
\end{proof}

In what follows we use the standard notation $\om_i=\sum_{a=1}^i \varepsilon_a$.
For a Young tableau $T$ the weight of $T$ is the linear combination $\sum_{a=1}^{r+1} m_a\varepsilon_a$, where $m_a$ is the
number of times $a$ shows up in $T$.

\begin{lem}\label{T}
Let us fix $i\ge j$. Then for any
$0\le l\le \min(j,r+1-i)$ there exists a two column Young tableau $T=(C,D)$ of shape $\omega_i+\omega_j$ with
the left column $C=(c_1<\dots <c_i)$ and right column $D=(d_1<\dots <d_j)$ such that $k(C,D)=l$ and the weight of
$T$ is equal to $\omega_{i+l}+\omega_{j-l}$.
\end{lem}
\begin{proof}
We need to fill $T=(C,D)$ with the numbers $1,\dots,r+1$ in such a way that each of the numbers $1,\dots,j-l$ shows up twice and
each of the numbers $j-l+1,\dots,i+l$ shows up once. Here is the rule: we consider the numbers $1,\dots, i+l$ and start
inserting these numbers into $T$ moving from  bottom to top in the following way: we first define
\[
c_i=i+l, c_{i-1}=i+l-1,\dots, c_j=j+l.
\]
Then we turn right, making a horizontal move, i.e. define $d_j=j+l-1$ and then go up, setting $d_{j-1}=j+l-2$. Then we
make a horizontal move by $c_{j-1}=j+l-3$ and a vertical step $c_{j-2}=j+l-4$ etc. until we make exactly $l$ horizontal moves.
Then we get $c_{j-l+1}=j-l+1$ or $d_{j-l+1}=j-l+1$. We finalize the procedure by setting
\[
c_{j-l-s+1}=d_{j-l-s+1}=j-l-s+1, \ s=0,\dots, j-l.
\]
\end{proof}

Here are examples of the tableaux $T$ from Lemma \ref{T} for $r=13$, $i=10$, $j=6$ with $l=0,1,2,3$:
\[
\begin{tabular}{cc}
1 & 1\\
2 & 2\\
3 & 3\\
4 & 4\\
5 & 5\\
6 & 6\\
7 & \\
8 & \\
9 & \\
10 &
\end{tabular} \qquad \qquad
\begin{tabular}{cc}
1 & 1\\
2 & 2\\
3 & 3\\
4 & 4\\
5 & 5\\
7 & 6\\
8 & \\
9 & \\
10 & \\
11 &
\end{tabular} \qquad \qquad
\begin{tabular}{cc}
1 & 1\\
2 & 2\\
3 & 3\\
4 & 4\\
5 & 6\\
8 & 7\\
9 & \\
10 & \\
11 & \\
12 &
\end{tabular} \qquad \qquad
\begin{tabular}{cc}
1 & 1\\
2 & 2\\
3 & 3\\
5 & 4\\
6 & 7\\
9 & 8\\
10 & \\
11 & \\
12 & \\
13 &
\end{tabular}
\]

\begin{rem}
The proof of Lemma \ref{Wmdec} below implies that the tableau $T$ constructed in the proof of Lemma \ref{T}
is unique with the desired properties.
\end{rem}

\begin{lem}\label{Wmdec}
For $1\le j\le i\le r$ the $q$-character of the local Weyl module $W_{\om_i+\om_j}$ is given by
\[
\ch W_{\om_i+\om_j} = \sum_{l=0}^{\min(j,r+1-i)}q^l\ch\ V_{\om_{j-l}+\om_{i+l}},
\]
where $\om_0=\om_{r+1}=0$.
\end{lem}
\begin{proof}
Recall that one has an isomorphism of $\fg$-modules $W_{\om_i+\om_j} = V_{\om_i}\T V_{\om_j}$.
The Littlewood-Richardson rule tells us that
\[
V_{\om_i}\T V_{\om_j}\simeq  \sum_{l=0}^{\min(j,r+1-i)} V_{\om_{j-l}+\om_{i+l}}.
\]
So we only need to show that the summand $V_{\om_{j-l}+\om_{i+l}}$ shows up at the level $l$ in
$W_{\om_i+\om_j}$. Theorem 3.21 of \cite{FM2} gives the following formula for the character of $W_{\om_i+\om_j}$
(here $x_a=\exp(\varepsilon_a)$):
\[
\ch\ W_{\om_i+\om_j}=\sum_{|I|=i, |J|=j} q^{k(I,J)}\prod_{a=1}^{l(I)} x_{i_a}\prod_{b=1}^{l(J)} x_{j_b}.
\]
We claim that for $1\le l\le \min(j,r+1-i)$
\begin{equation}\label{tensor}
\ch\ V_{\om_{j-l}+\om_{i+l}}=\sum_{\substack{|I|=i, |J|=j\\ k(I,J)=l}} \prod_{a=1}^{l(I)} x_{i_a}\prod_{b=1}^{l(J)} x_{j_b},
\end{equation}
(for example, \eqref{tensor} for $l=0$ simply means that semi-standard tableaux form a basis for the module $V_{\om_i+\om_j}$).
In fact, the right hand side of \eqref{tensor} is a character of a summand of the tensor product $V_{\om_i}\T V_{\om_j}$.
Thanks to Lemma \ref{T} the right hand side of \eqref{tensor} contains a term corresponding to the highest
weight vector of $V_{\omega_{i+l}+\om_{j-l}}$. This implies \eqref{tensor}.
\end{proof}

Recall that the algebra $\gW^*$ is generated by its fundamental part $\bigoplus_{a=1}^r \gW^*_{\om_a}$.
Proposition below claims that the vertex operator approach produces the generating set of relations in type $A$.
It is tempting to conjecture that the same result hold in all types. Our argument is based on \cite{FM2},
which treats type $A$ only. For other types it is not even known if the ideal of relations is generated by quadratic part.

\begin{prop}
The relations described in Theorem \ref{descr} generate the ideal of relations of $\gW^*$.
\end{prop}
\begin{proof}
Corollary 3.22 and Proposition 2.9 a), \cite{FM2} state that the ideal of relations of $\gW^*$
is generated by its quadratic part and the generating set of relations can be described as follows:
given a non semi-standard two-column Young tableaux $T=(I,J)$ of shape $\om_i+\om_j$ one produces a relation
in the coordinate ring $\bigoplus_{\la\in P_+} V_\la^*$ of the form $R(I,J)=X_IX_J+\sum_{I',J'} \pm X_{I'}X_{J'}$
(here the Pl\"ucker coordinates $X_I$ are considered as elements of $V^*_{\om_{l(I)}}$).
In other words, $R(I,J)$ belongs to the kernel of the map $V^*_{\om_i}\T V^*_{\om_j}\to V^*_{\om_i+\om_j}$.
Now the relation $R(I,J)$ produces $k(I,J)$ series of relations of the form
\begin{equation}\label{ideal}
\frac{\partial_s X_I(z)}{\partial z^s}X_J(z)+\sum_{I',J'} \pm X_{I'}(z)X_{J'}(z),\ s<k(I,J),
\end{equation}
where $X_I(z)=\sum_{k\ge 0} (X_It^{-k})z^{k}$ and $X_It^{-k}$ is considered as an element of
$W^*_{\om_{l(I)}}[t^{-1}]\simeq \gW^*_{\om_{l(I)}}$.
Corollary 3.22 and Proposition 2.9 a), \cite{FM2} state that the coefficients of the relations above generate
the ideal of relations in $\gW^*$.

For a positive integer $l$ we consider a map
\[
\varphi_l: W^*_{\om_i}\T W^*_{\om_j}\to\gW^*_{\om_i+\om_j},\ \varphi_l (A,B)=[A_{(m_{i,j}-l)}B]
\]
where as usual the notation $[A_{(m_{i,j}-l)}B]$ is used to denote the class of the element $A_{(m_{i,j}-l)}B$
in ${\rm gr} G_{\om_i+\om_j}$.
We are interested in the kernel of $\varphi_l$. Let $\gW^*_{\om_i+\om_j}=\sum_{s\ge 0} \gW^*_{\om_i+\om_j}[s]$
be the decomposition with respect to the $q$-grading. Then $\varphi_l$ is a $\fg$-homomorphism from
$V^*_{\om_i}\T V^*_{\om_j}$ to $\gW^*_{\om_i+\om_j}$ thanks to Corollary \ref{26} (recall $W_{\om_i}\simeq V_{\om_i}$).
Thanks to Lemma \ref{Wmdec} the decomposition of the $\fg$-module $\gW^*_{\om^*_i+\om^*_j}[s]$
into irreducible components  does not contain direct summands of the form $V^*_{\om_{i+l}+\om_{j-l}}$
for $l>s$. We conclude that
\begin{equation}\label{kernel}
\ker \varphi_l\supset V^*_{\om_{i+l}+\om_{j-l}} \oplus V^*_{\om_{i+l+1}+\om_{j-l-1}}\oplus\dots \oplus V^*_{\om_{i+\min(j,r+1-i)}
+ \om_{j-\min(j,r+1-i)}}.
\end{equation}
According to Theorem \ref{descr} for each $l\le \min(j,r+1-i)$ we obtain the relations of the form
\begin{equation}\label{vrel}
\sum_m (\partial_z^{l-1}A^m(z))B^m(z), \ s\le l
\end{equation}
for each element $\sum_m A^m\T B^m\in V^*_{\om_{i+l}+\om_{j-l}} \subset V^*_{\om_i}\T V^*_{\om_j}$.

We claim that all the relations \eqref{ideal} can be obtained in this way. Indeed, let us consider relation
\eqref{ideal} attached to the two column tableau from Lemma \ref{T}. Then this relation is a highest weight vector
of $V^*_{\om_{i+l}+\om_{j-l}}\subset V^*_{\om_i}\T V^*_{\om_j}$ and thus is contained in \eqref{vrel}.
Now we use the $\fg$-invariance of the relations \eqref{ideal} and \eqref{vrel}.
\end{proof}

\section*{Acknowledgments}
The work on Sections 1,2,3 was partially supported by the Russian Academic Excellence Project '5-100'.
The work on Sections 4,5,6 was partially supported by the grant RSF-DFG 16-41-01013.

\end{document}